\documentclass[12pt]{amsart}
\usepackage{amssymb,amsfonts,amsmath,amsopn,amstext,amscd,latexsym, amsthm, enumerate,mathrsfs}

\usepackage{color}
\usepackage{xcolor}
\usepackage[]{hyperref}
\hypersetup{
     colorlinks   = false,
     citecolor    = blue
}

\usepackage[margin=30mm]{geometry}
\usepackage{bbm}
\usepackage[all,cmtip]{xy}

\usepackage{enumerate}
\usepackage[shortlabels]{enumitem}

\newtheorem{theorem}{Theorem}[section]

\newtheorem*{thmA}{Theorem~ A}
\newtheorem*{corB}{Corollary~ B}

\newtheorem{lem}[theorem]{Lemma}

\theoremstyle{remark}

\DeclareMathOperator{\Gal}{Gal}
\DeclareMathOperator{\Aut}{Aut}

\DeclareMathOperator{\Irr}{Irr}

\DeclareMathOperator{\GL}{GL}

\DeclareMathOperator{\PSL}{PSL}

\newcommand{\cd}{{\mathrm {cd}}}

\renewcommand{\mod}{\bmod \,}

\numberwithin{equation}{section}

\newcommand{\Out}{{\mathrm {Out}}}

\begin{document}

\title[Character degrees of extensions of $^2B_2(q^2)$]{Character degrees of extensions of the Suzuki groups $^2B_2(q^2)$}

\author[M.~Ghaffarzadeh]{Mehdi Ghaffarzadeh}
\address{Department of Mathematics, Khoy Branch, Islamic Azad University, Khoy, Iran}
\email{gh.ghaffarzadeh@iaukhoy.ac.ir}

\thanks{}



\date{April 14, 2016}


\begin{abstract}
Let $S$ be a Suzuki group $^2B_2(q^2)$, where $q^2=2^{2f+1}$, $f\geqslant 1$. In this paper, we determine the degrees of the ordinary complex  irreducible  characters of every group $G$ such that $S\leqslant G\leqslant \Aut(S)$.
\end{abstract}

\maketitle
\section{Introduction}

Given a finite group $G$, let $\cd(G)=\{\chi(1) \mid \chi\in\Irr(G)\}$ be the set of degrees of the ordinary complex irreducible characters of $G$. A common approach in the studying of nonsolvable groups with a given property on irreducible character degrees begins by examining the property on simple and almost simple groups. Among these, depending on the given property, the most work is done generally on simple groups $S$ with few character degrees and even more so, on groups $G$ with $S\leqslant G\leqslant \Aut(S)$ such that a few characters of $S$ are extendible to $G$, (see \cite{HZ, LWodd, MMnon}). 

The $2$-dimensional projective special linear groups $\PSL_2(q)$, $q>3$ with $4$ or $5$ degrees and the Suzuki groups $^2B_2(q^2)$ with $6$ degrees are two families of nonabelian simple groups which have the least number of character degrees. In addition, if $S$ is a simple group of Lie type, then by Theorems 2.4 and 2.5 of \cite{Mal}, all unipotent characters of $S$ apart from  a few cases extend to $\Aut(S)$. By using \cite[Chapter  13]{Carter} and \cite{Mal} for  simple groups of Lie type and the Atlas \cite{Atlas} and hook-partitions for sporadic and alternating groups, it follows that the families $\PSL_2(q)$ and $^2B_2(q^2)$ are among the nonabelian simple groups with a small number of degrees extendible to $\Aut(S)$. All of these emphasize to have an explicit result on the set of degrees of the irreducible characters of every almost simple group $G$ for which $S=\PSL_2(q)$ or $^2B_2(q^2)$. For $S=\PSL_2(q)$, this has been done in \cite{White}. In this paper, we consider the Suzuki groups $^2B_2(q^2)$, where $q^2=2^{2f+1}$, $f\geqslant 1$.

We will use the notation of \cite{Suz} for the Suzuki groups,  in which instead of $q$ we write $q^2$. By Theorem 5 of \cite{Suz}, the character degree set of $^2B_2(q^2)$ is 
\[\cd(^2B_2(q^2))=\{1, q^4, q^4+1, (q^2-r+1)(q^2-1), (q^2+r+1)(q^2-1), r(q^2-1)/2\},\] where $ r=2^{f+1}$.

Our main result is the following:

\begin{thmA}\label{main}
Let $S=  {^2B_2(q^2)}$, where $q^2=2^{2f+1}$, $f\geqslant 1$ and let $S\leqslant G\leqslant \Aut(S)$ with $|G : S|=d$. Then the set of irreducible character degrees of $G$ is 
\[\cd(G)=\{1, q^4, r(q^2-1)/2\}\cup\{(q^4+1)a, (q^2-r+1)(q^2-1)b, (q^2+r+1)(q^2-1)c : a, b, c\mid d\},\]
with the following exceptions:
\begin{itemize}
\item[{\rm (i)}] If $G=\Aut(S)$, then $a\not=1$,
\item[{\rm (ii)}] If $f\equiv 1$ or $2\,\, (\mod 4)$, and $G=\Aut(S)$, then $b\neq 1$ and $c\neq 3$, 
\item[{\rm (iii)}] If $f\equiv 0$ or $3 \,\,(\mod 4)$, and $G=\Aut(S)$, then $b\neq 3$ and $c\neq 1$.
\end{itemize}
\end{thmA}
 As an immediate consequence of Theorem A, we obtain the following.
\begin{corB}\label{corB}
Let $S=  {^2B_2(q^2)}$, where $q^2=2^{2f+1}$, $f> 1$ and let $S< G\leqslant \Aut(S)$ with $|G : S|=d$. Then
 $|\cd(G)|\geqslant 7$, and if $d$ is not a  prime, then $|\cd(G)|\geqslant 9$. 
 \end{corB}
For a normal subgroup $S$ of a finite group $G$, we denote by $\cd(G\mid S)$, the set of degrees of irreducible characters of $G$ whose kernels do not contain $S$. Observe that Corollary B can be used to improve the conclusions of \cite{HZ}, on finite groups $G$ with a nonsolvable normal subgroup $S$ such that $|\cd(G\mid S)|\leqslant 5$. See \cite[Theorems ~ A, B and Corollary C]{HZ}.

\medskip
At the end of this section, we sketch our proof of Theorem A as follows. First, note that $S\unlhd G$, so 
\[\cd(G)=\bigcup_{\theta\in\Irr(S)}\cd(G\mid \theta),\]
in which $\cd(G\mid\theta)=\{\chi(1) \mid \chi\in\Irr(G), [\chi_S, \theta]\neq 0\}$. Also, the group of outer automorphisms   of $S$ is
 cyclic. Thus by \cite[Corollary~11.22]{Isa}, every $\theta\in\Irr(S)$ extends to its stabilizer $I_G(\theta)$ in $G$ and so  Gallagher's theorem \cite[Corollary~ 6.17]{Isa} yields that $\cd(I_G(\theta)\mid \theta)=\{\theta(1)\}$. 
 Using Clifford's theorem \cite[Theorem~ 6.11]{Isa}, we obtain $\cd(G\mid \theta)=\{|G : I_G(\theta)|\theta(1)\}$. So what is important here, is to find the stabilizer of any $\theta\in\Irr(S)$ in $G$. To do this we need several number theoretic results on divisors of $q^4+1$ and $q^2\pm r+1$, which will be provided in section \ref{sec2}.


\section{Preliminaries}\label{sec2}
We first establish some notation which will remain fixed throughout this paper. Simultaneously, we mention some facts from \cite{Suz} about the Suzuki groups which will be needed in this paper. Let $f$ be a positive integer, let $q^2=2^{2f+1}$ and $r=2^{f+1}$. We will denote the Suzuki group $^2B_2(q^2)$ by $S$. Note that $S$  is a simple group of order $(q^4+1)q^4(q^2-1)$ and its order is not divisible by $3$.  We may regard $S$ as a subgroup of $\GL_4(q^2)$. Using the notation of \cite{Suz}, we know that $S$ has $q^2+3$ conjugacy classes of elements, consisting of a class of involutions with a representative $\sigma$, two classes of elements of order $4$ with representatives $\rho$, $\rho^{-1}$ and $q^2$ classes of semisimple elements. Also, by \cite[Theorem~ 4]{Suz}, any  semisimple element of $S$ is conjugate to an element of the cyclic subgroups $A_0$, $A_1$, $A_2$ of $S$.

\medskip
In order to determine which subgroups of $\Aut(S)$ appear as stabilizer of irreducible characters of $S$, we first consider the action of outer automorphisms of $S$ on conjugacy classes. According to \cite[Theorem~ 11]{Suz}, the group of outer automorphisms of $S$ is cyclic of odd order $2f+1$ and is generated by a field automorphism. To be more precise, let $\mathcal{G}=\Gal({{\mathbb{F}}_{q^2}}/{\mathbb{F}_2})$. Note that $\mathcal{G}$ is generated by the Frobenius automorphism $\overline{\varphi}$, where $\overline{\varphi}(\zeta)={\zeta}^2$ for all $\zeta\in\mathcal{G}$. Then $\Out(S)$ is generated by the outer automorphism $\varphi$ of $S$ induced by $\overline{\varphi}$. By \cite{Suz}, we obtain the following lemma on the action of $\varphi$ on conjugacy classes of $S$. 

\begin{lem}\label{2-1}
Assume notation as above and let $n$ be an integer. The automorphism $\varphi$ fixes the conjugacy classes of $\sigma$, $\rho$, $\rho^{-1}$ in $S$. Also, if $x$ is a semisimple element of $S$, then $\varphi$ sends the conjugacy class of $x$ to the class of $x^2$, and $\varphi^n$ sends the class of $x$ to the class of $x^{2^n}$.
\end{lem}

We now consider the action of the field automorphism $\varphi$ on the irreducible characters of $S$. By \cite[Section~ 17]{Suz}, the nonlinear irreducible characters of $S$ along with their degrees and the number of characters of each degree are as in the following table: 
\begin{center}
\begin{tabular}{c@{ \hskip 0.5in }c@{\hskip 0.5in}c}
 $X$ &  $q^4$ & $1$ \\ 
 $X_i$ &  $q^4+1$ & $q^2/2-1$ \\ 
 $Y_j$ &  $(q^2-r+1)(q^2-1)$ & $(q^2+r)/4$ \\ 
 $Z_k$ &  $(q^2+r+1)(q^2-1)$ & $(q^2-r)/4$ \\ 
 $W_l$ &  $r(q^2-1)/2$ & $2$
\end{tabular}
\end{center}

The characters of $S$ of degrees $1$ and $q^4$ are unique, so they are invariant under $\varphi$. Also, the characters $W_l$ with $l=1,2$ are equal on all classes of semisimple elements, (see \cite[Theorem ~ 13]{Suz}). So they are invariant under $\varphi$ by Lemma \ref{2-1}. Thus by \cite[Corollaries ~ 11.22 and 6.17]{Isa}, we obtain the following result.
\begin{lem}\label{2-2}
All characters of $S$ of degrees $1$, $q^4$ and $r(q^2-1)/2$ are invariant under $\varphi$ and each extends to $2f+1$ distinct irreducible characters of $\Aut(S)$.
\end{lem}
It remains to determine the subgroups of $\Aut(S)$ that occur as stabilizers of characters $X_i$, $Y_j$ and $Z_k$ of $S$. For this, we collect in this section, several number theoretic lemmas, which will be used frequently in our proofs. 
\begin{lem}{\cite[Lemma 4.9]{White}{\rm)}}\label{2-3}
If $p$ is a prime and $m$, $n$ are positive integers such that $m\mid n$, then
\begin{itemize}
\item[{\rm (i)}] $(p^n-1, p^m-1)=p^m-1$, 
\vspace{0.3cm}
\item[{\rm (ii)}] $(p^n-1, p^m+1) = \left\{
\begin{array}{ll}
(p-1, 2) & \text{if $n/m$ is odd},\\
p^m+1 & \text{if $n/m$ is even},
\end{array} \right. $ 
\vspace{0.3cm}
\item[{\rm (iii)}] $(p^n+1, p^m-1)=(p-1, 2)$,
\vspace{0.3cm}
\item[{\rm (iv)}] $(p^n+1, p^m+1) = \left\{
\begin{array}{ll}
p^m+1 & \text{if $n/m$ is odd},\\
(p-1, 2) & \text{if $n/m$ is even}.
\end{array} \right.$  
\end{itemize}
\end{lem}
\begin{lem}\label{2-4}
If $n$ is a  proper positive divisor of $2f+1$, then 
\begin{itemize}
\item[{\rm (i)}] $(q^4+1, 2^n\pm 1)=1$, 
\vspace{0.3cm}
\item[{\rm (ii)}]  $(q^4+1, q^2-2^n) = \left\{
\begin{array}{ll}
2^{2n}+1 & \text{if $2f+1\equiv n \,(\mod 4)$}, \\
1 & \text{otherwise}, 
\end{array} \right.$
\vspace{0.3cm}
\item[{\rm (iii)}] $(q^4+1, q^2+2^n) = \left\{
\begin{array}{ll}
2^{2n}+1 & \text{if  $2f+1\equiv -n \,(\mod 4)$}, \\
1 & \text{otherwise}. 
\end{array} \right.$
\end{itemize}
\end{lem}
\begin{proof}
Part (i) follows from Lemma \ref{2-3}. Let $d=(q^4+1, q^2-2^n)$. Since $d$ is odd, we have $d\mid 2^{2f-n+1}-1$ and hence $d\mid 2^{4f-2n+2}-1$. Also $d\mid 2^{4f+2}+1$, thus  $d\mid 2^{4f-2n+2}(2^{2n}+1)$ and so $d\mid 2^{2n}+1$. On the other hand, by Lemma \ref{2-3}, we have $2^{2n}+1\mid q^4+1$ and
\begin{equation*}
(q^2-2^n, 2^{2n}+1)=(2^{2f-n+1}-1, 2^{2n}+1) = \left\{
\begin{array}{ll}
2^{2n}+1 & \text{if  $(2f-n+1)/{2n}$ is even}, \\
1 & \text{otherwise}.
\end{array} \right.
\end{equation*}
Thus,  if $2f+1\equiv n\, (\mod  4)$, then $2^{2n}+1\mid q^2-2^n$ and so $d=2^{2n}+1$, while if $2f+1\not\equiv n \,(\mod 4)$, then $(q^2-2^n, 2^{2n}+1)=1$ and $d=1$. Hence (ii) follows. A similar argument applies to (iii). Now the proof is complete. 
\end{proof}
In the following lemma, we shall use $p^n\top d$ to denote that $p^n$ is the highest power of prime number $p$ which divides $d$.
\begin{lem}\label{2-5}
Let $n$ be a proper positive divisor of $2f+1$. Then 
\begin{itemize}
\item[{\rm (i)}]  $(q^2\pm r+1, q^2-2^n) = \left\{
\begin{array}{ll}
2^{n}\pm 2^{(n+1)/2}+1 & \text{if $8\mid 2f-n+1$},\\
2^{n}\mp 2^{(n+1)/2}+1 & \text{if  $4\top 2f-n+1$},\\
1 & \text{otherwise,} 
\end{array} \right.$
\vspace{0.3cm}
\item[{\rm (ii)}] $(q^2\pm r+1, q^2+2^n) = \left\{
\begin{array}{ll}
2^{n}\pm 2^{(n+1)/2}+1 & \text{if  $8\mid 2f+n+1$},\\
2^{n}\mp 2^{(n+1)/2}+1 & \text{if  $4\top 2f+n+1$},\\
1 & \text{otherwise.} 
\end{array} \right.$
\end{itemize}
\end{lem}
\begin{proof}
(i)  Observe first that 
\begin{equation*}
q^4+1 = (q^2-r+1)(q^2+r+1). 
\end{equation*}
If $2f+1\not\equiv n \,(\mod 4)$, then Lemma \ref{2-4}(ii) yields 
 $(q^2\pm r+1, q^2-2^n)=1$. So we may assume that $2f+1\equiv n \,(\mod 4)$. If $2f+1\equiv n \,(\mod 8)$, then $2^{2n }+1\mid 2^{(2f-n+1)/2}-1$ by Lemma \ref{2-3}(ii) and hence 
 \begin{equation*}\label{eq1}
(2^{(2f-n+1)/2}-1)(2^{(2f+n+1)/2}-1)\equiv 0 \,\,\,(\mod 2^{2n}+1). 
\end{equation*}
Thus
\begin{equation*}
q^2\pm r+1\equiv 2^{(2f-n+1)/2}(2^{n}\pm 2^{(n+1)/2}+1) \,\,\,\, (\mod 2^{2n}+1). 
\end{equation*}
Since $2^{2n}+1=(2^n-2^{(n+1)/2}+1)(2^n+2^{(n+1)/2}+1)$, we have
\begin{equation*}
 2^n+2^{(n+1)/2}+1\mid q^2+r+1 \quad \text{and}\quad 2^n-2^{(n+1)/2}+1\mid q^2-r+1.
\end{equation*}
Now, Lemma \ref{2-4}(ii) applies to yield that 
\begin{equation*}
(q^2\pm r+1, q^2-2^n)=2^n\pm 2^{(n+1)/2}+1.
\end{equation*}
Finally, if $4\top 2f-n+1$, then $2^{2n}+1\mid 2^{(2f-n+1)/2}+1$ by Lemma \ref{2-3}(iv). Thus 
\begin{equation*}\label{eq2}
(2^{(2f-n+1)/2}+1)(2^{(2f+n+1)/2}+1)\equiv 0 \,\,\,(\mod 2^{2n}+1).
\end{equation*}
Therefore,
\begin{equation*}
q^2\pm r+1\equiv -2^{(2f-n+1)/2}(2^n\mp 2^{(n+1)/2}+1) \,\,\,(\mod 2^{2n}+1).
\end{equation*}
Agian, because of $2^{2n}+1=(2^n-2^{(n+1)/2}+1)(2^n+2^{(n+1)/2}+1)$, we deduce from Lemma \ref{2-4}(ii) that 
\begin{equation*}
(q^2\pm r+1, q^2-2^n)=2^n\mp 2^{(n+1)/2}+1.
\end{equation*}
Hence (i) follows.

(ii) Note that if $8\mid 2f+n+1$, then $2^{2n}+1\mid 2^{(2f+n+1)/2}-1$ by Lemma \ref{2-3}. Also, if $4\top 2f+n+1$, then $2^{2n}+1\mid 2^{(2f+n+1)/2}+1$, so we can apply the above argument for (ii). The proof is now complete.
\end{proof}
\begin{lem}\label{2-6}
Let $m$, $n$ be distinct proper positive divisors of $2f+1$ with $m\mid n$. Set 
\begin{equation*}
d_1=(q^2+r+1, q^2-2^n) \quad \text{or} \quad (q^2+r+1, q^2+2^n),
\end{equation*}
\begin{equation*}
d_2=(q^2+r+1, q^2-2^m) \quad \text{or} \quad (q^2+r+1, q^2+2^m).
\end{equation*}
If $d_1>1$, then $d_1\neq d_2$ except when $m=1$, $n=3$ and one of the following holds:
\begin{itemize}
\item[{\rm (i)}]  $f\equiv 0 \,(\mod 4), \quad d_1=(q^2+r+1, q^2+8),\quad d_2=(q^2+r+1, q^2-2)$,
\item[{\rm (ii)}] $f\equiv 3\, (\mod 4), \quad d_1=(q^2+r+1, q^2-8),\quad d_2=(q^2+r+1, q^2+2)$.
\end{itemize}
Similarly, replacing $q^2+r+1$ above, by $q^2-r+1$ yields that $d_1\neq d_2$ except when $m=1$, $n=3$ and one of the following holds: 
\begin{itemize}
\item[{\rm (iii)}]  $f\equiv 1 \,(\mod 4), \quad d_1=(q^2-r+1, q^2-8),\quad d_2=(q^2-r+1, q^2+2)$,
\item[{\rm (iv)}] $f\equiv 3 \,(\mod 4), \quad d_1=(q^2-r+1, q^2+2),\quad d_2=(q^2-r+1, q^2-2)$.
\end{itemize}

In cases {\rm (i)}-{\rm (iv)}, we have $d_1=d_2=5$.
\end{lem}

\begin{proof}
Note that exactly one of $2f\pm n+1$ is divisible by $4$. So only one of $(q^2+r+1, q^2\pm 2^n)$ is greater than $1$ by Lemma \ref{2-5}. The same is true for $(q^2-r+1, q^2\pm 2^n)$. Also, if $n>3$, then clearly $d_1>5$, while if $n=3$, then $d_1=5$ or $13$. By  comparing the values of $d_1$ and $d_2$, that is the numbers $2^n\pm 2^{(n+1)/2}+1$ with $1$, $2^m\pm 2^{(m+1)/2}+1$, it is relevant to see that $d_1\neq d_2$ with the exception that $m=1$, $n=3$ and $d_1=d_2=5$. By Lemma \ref{2-5}, $d_1=d_2=5$ if and only if one of (i)-(iv), given above holds. 
\end{proof}
The following lemma will be useful when working with characters $Y_j$, $Z_k$ of $S$.
\begin{lem}\label{2-7}
Let $\varepsilon$ be a complex $n$th root of unity and let $i$, $j$, $k$ be integers such that $k^2\equiv -1 \,(\mod n)$. Then 
\begin{equation}\label{eq3}\tag{1}
\varepsilon^{il}+\varepsilon^{-il}+\varepsilon^{ilk}+\varepsilon^{-ilk}=\varepsilon^{jl}+\varepsilon^{-jl}+\varepsilon^{jlk}+\varepsilon^{-jlk}, \quad \text{where}\,\,\, l=1, k-1, 
\end{equation} 
if and only if one of the congruents $i\equiv \pm j\, (\mod n)$ or $i\equiv \pm jk \,(\mod  n)$ holds.
\end{lem}

\begin{proof}
Suppose that (\ref{eq3}) holds for $l=1, k-1$. Observe that 
\begin{equation*}
\begin{split}
(\varepsilon^i+&\varepsilon^{-i}-\varepsilon^{j}-\varepsilon^{-j})(\varepsilon^{i}+\varepsilon^{-i}-\varepsilon^{jk}-\varepsilon^{-jk}) =
 2+\varepsilon^{2i}+\varepsilon^{-2i}\\
 & -(\varepsilon^i+\varepsilon^{-i})(\varepsilon^j+\varepsilon^{-j}+\varepsilon^{jk}+\varepsilon^{-jk})
   +\varepsilon^{j(k-1)}+\varepsilon^{-j(k-1)}+\varepsilon^{j(k+1)}+\varepsilon^{-j(k+1)}.
 \end{split}
\end{equation*}
Replacing $\varepsilon^j+\varepsilon^{-j}+\varepsilon^{jk}+\varepsilon^{-jk}$ above, by $\varepsilon^i+\varepsilon^{-i}+\varepsilon^{ik}+\varepsilon^{-ik}$ yields 
\begin{equation*}
\begin{split}
(\varepsilon^i+ & \varepsilon^{-i}-\varepsilon^{j}-\varepsilon^{-j})(\varepsilon^{i}+\varepsilon^{-i}-\varepsilon^{jk}-\varepsilon^{-jk}) \\
 & =-\varepsilon^{i(k-1)}-\varepsilon^{-i(k-1)} -\varepsilon^{i(k+1)}-\varepsilon^{-i(k+1)} +\varepsilon^{j(k-1)}+\varepsilon^{-j(k-1)}+\varepsilon^{j(k+1)}+\varepsilon^{-j(k+1)}\\
 &=0,
 \end{split}
\end{equation*}
where the last equality follows from (\ref{eq3}) with $l=k-1$, since $k^2\equiv -1\, (\mod n)$. Therefore,
\begin{equation*}
\varepsilon^i+ \varepsilon^{-i}=\varepsilon^{j}+\varepsilon^{-j}\quad \text{or}\quad \varepsilon^{i}+\varepsilon^{-i}=\varepsilon^{jk}+\varepsilon^{-jk}.
\end{equation*}
Now, Lemma 4.17 of \cite{White} yields that $i\equiv \pm j \,(\mod n)$ or $ i\equiv \pm jk \,(\mod n)$. Conversely, since $k^2\equiv -1 \,(\mod n)$,  it is  clear that each of the above congruents yields (\ref{eq3}) for all integers $l$.  
\end{proof}

\section{Stabilizers of characters of degrees $q^4+1$ and $q^2\pm r+1$}

In this section, we shall determine the subgroups of $\Aut (S)$ that occur as stabilizers of characters $X_i, Y_j$, $Z_k$  of $S$. According to \cite[Theorem ~ 9]{Suz}, the group $S$ contains three cyclic subgroups $A_0, A_1$ and  $ A_2$ of order $q^2-1, q^2+r+1$ and  $q^2-r+1$, respectively. Also, by \cite[Theorem ~13]{Suz}, the characters $X_i$ are equal everywhere except on $A_0-\{1\}$. The  same property is valid for characters $Y_j$ and  $Z_k$ in corresponding with subgroups $A_1$ and $ A_2$, respectively.

First we consider the characters $X_i$, where $1\leqslant i\leqslant q^2/2-1$. Let $\varepsilon_0$ be a primitive $(q^2-1)$th root of unity. If $\xi_0$ is a generator of $A_0$, then
\begin{equation*}
X_i({\xi_0}^l)=\varepsilon_0^{il}+\varepsilon_0^{-il},
\end{equation*}
for all $l$, $1\leqslant l\leqslant q^2-1$. By Lemma \ref{2-1}, the action of a power $\varphi^n$ of $\varphi$ on $\xi_0^l$ is as follows:
\begin{equation*}
{X_i}^{\varphi^n}(\xi_0^l)=X_i((\xi_0^l)^{\varphi^{-n}})=X_i(\xi_0^{-l2^n})=\varepsilon_0^{il2^n}+\varepsilon_0^{-il2^n}.
\end{equation*}
\begin{lem}\label{3-1}
Let $i, n$ be  positive integers such that $1\leqslant i\leqslant q^2/2-1$ and  $n\mid 2f+1$. 
\begin{itemize}
\item[{\rm (i)}]  The character $X_i$ is invariant under $\varphi^n$ if and only if $q^2-1\mid (2^n-1)i$. 
\item[{\rm (ii)}] Let $N=\langle\varphi^n\rangle$ and set $i=(q^2-1)/(2^n-1)$. Then $N$ is the stabilizer of $X_i$ in $\langle \varphi\rangle$ unless when $n=1$, in which case $N$ does not stabilize any $X_i$.
\end{itemize}
\end{lem}
\begin{proof}
(i) By Lemma \ref{2-1} and the above discussion, we know that $X_i$ is invariant under  $\varphi^n$ if and only if  
$X_i(\xi_0^l)=X_i^{\varphi^n}(\xi_0^l)$, for all $l$, $1\leqslant l\leqslant q^2-1$, hence if and only if 
\[\varepsilon_0^{il}+\varepsilon_0^{-il}=\varepsilon_0^{il2^n}+\varepsilon_0^{-il2^n},\]
for all $l$. In particular, by taking $l=1$, we obtain $i2^n\equiv \pm i \,(\mod q^2-1)$, see \cite[Lemma~ 4.7]{White}. Note  that the congruent $i2^n\equiv -i \,(\mod q^2-1)$ is impossible, otherwise Lemma \ref{2-3}(ii) yields that $q^2-1\mid i$, a contradiction to $1\leqslant i \leqslant q^2/2-1$. Hence  $X_i$ is invariant under $\varphi^n$ if and only if $q^2-1\mid (2^n-1)i$, as claimed.
\\

(ii) Note that if $n=1$, then part (i) implies there is no $X_i$ stabilized by $N=\langle \varphi\rangle$. So we may assume that $n>1$. It follows that 
\[i=\frac{q^2-1}{2^n-1}\leqslant \frac{q^2-1}{3}<\frac{q^2}{2}-1.\]
By Lemma \ref{2-3}, $i$ is an integer, hence $X_i\in\Irr(S)$. Moreover, we have $q^2-1= (2^n-1)i$, so  by part (i), $X_i$ is stabilized by $N$. If the stabilizer, $T$, of $X_i$ in $\langle \varphi\rangle$ properly contains $N$, then $T=\langle \varphi^t\rangle$ for some divisor $t$ of $n$ with $1\leqslant t< n$. By part (i), we have $q^2-1\mid (2^t-1)i$.  It follows that  $2^n-1\mid 2^t-1$, a contradiction, since $t<n$. Therefore, $N$ is the stabilizer of $X_i$ in $\langle\varphi\rangle$.  
\end{proof}

We next consider the characters $Y_j$, where $1\leqslant j\leqslant (q^2+r)/4$. Let $\varepsilon_1$ be a primitive $(q^2+r+1)$th root of unity. If $\xi_1$ is a generator of $A_1$, then
\begin{equation*}
Y_j({\xi_1}^l)=-(\varepsilon_1^{jl}+\varepsilon_1^{-jl}+\varepsilon_1^{jlq^2}+\varepsilon_1^{-jlq^2}),
\end{equation*}
for all $l$, $1\leqslant l\leqslant q^2+r+1$. Note that by \cite[Theorem ~13]{Suz}, the characters $Y_j$ are equal everywhere except on $A_1-\{1\}$.

\begin{lem}\label{3-2}
Let $j, n$ be  positive integers such that $1\leqslant j\leqslant (q^2+r)/2$ and  $n\mid 2f+1$.  Then the character $Y_j$ is invariant under $\varphi^n$ if and only if 
\begin{equation*}
q^2+r+1\mid (q^2-2^n)j\qquad\text{or}\qquad q^2+r+1\mid (q^2+2^n)j.
\end{equation*}
\end{lem}

\begin{proof}
The character $Y_j$ is invariant under $\varphi^n$ if and only if 
\[Y_j(\xi_1^l)=Y_j^{\varphi^n}(\xi_1^l)=Y_j(\xi_1^{-l2^{n}}),\]
for all $l$, $1\leqslant l\leqslant q^2+r+1$, hence  if and only if 
\[\varepsilon_1^{jl}+\varepsilon_1^{-jl}+\varepsilon_1^{jlq^2}+\varepsilon_1^{-jlq^2}=\varepsilon_1^{jl2^n}+\varepsilon_1^{-jl2^n}+\varepsilon_1^{jlq^22^n}+\varepsilon_1^{-jlq^22^n},\]
for all $l$. Since $q^4\equiv -1 \,(\mod q^2+r+1)$, so by Lemma \ref{2-7}, this holds if and only if 
\[ j2^n\equiv \pm j \,\,(\mod q^2+r+1)\qquad \text{or}\qquad jq^2\equiv \pm j2^n \,\,(\mod q^2+r+1).\]
The congruents $j2^n\equiv\pm j \,(\mod q^2+r+1)$ are impossible, otherwise Lemma \ref{2-4}(i) yields $q^2+r+1\mid j$, a contradiction to $1\leqslant j \leqslant (q^2+r)/4$. Hence the result follows. 
\end{proof}

\begin{lem}\label{3-3}
Let $n$ be a positive divisor of $2f+1$. Set $N=\langle\varphi^n\rangle.$ Then $N$ is the stabilizer in $\langle\varphi\rangle$ of some $Y_j$ unless one of the following cases occur.
\begin{itemize}
\item[{\rm (i)}]  $n=1$ and $f\equiv 1$ or $2\, (\mod 4)$, 
\item[{\rm (ii)}] $n=3$ and $f\equiv 0$ or $3 \,(\mod 4)$.
\end{itemize}
In cases {\rm (i)} and {\rm (ii)}, $N$ does not stabilize any $Y_j$.
\end{lem}

\begin{proof}
If $Y_1$ is invariant under $\varphi^n$, then $q^2+r+1\mid q^2-2^n$ or $q^2+r+1\mid q^2+2^n$ by Lemma \ref{3-2}. (It is relevant to see that if $n<2f+1$, then none of the numbers $q^2\pm r+1$ divide $q^2-2^n$ or $q^2+2^n$.)  Hence we must have $n=2f+1$. It follows that $N=S$ is the stabilizer of $Y_1$ in $\langle \varphi\rangle$.

Assume then that $n<2f+1$. If $n=1$ and $f\equiv 1$ or $2 \,(\mod 4)$, then Lemma \ref{2-5} yields $(q^2+r+1, q^2\pm2)=1$. In this case, we know  by Lemma \ref{3-2} that $N=\langle\varphi\rangle$ does not stabilize any $Y_j$. Assume now that $n=3$ and $f\equiv 0$ or $3 \,(\mod 4)$. If $N=\langle\varphi^3\rangle$ stabilizes some $Y_j$, then Lemma \ref{3-2} yields that 
\[q^2+r+1\mid (q^2-8)j \quad \text{or}\quad q^2+r+1\mid (q^2+8)j.\]
By Lemma \ref{2-6}, we have 
\[(q^2+r+1, q^2-8)=(q^2+r+1, q^2+2)=5\quad\text{or}\quad (q^2+r+1, q^2+8)=(q^2+r+1, q^2-2)=5,\]
hence
\[q^2+r+1\mid (q^2+2)j \quad \text{or} \quad q^2+r+1\mid (q^2-2)j.\]
Thus $\langle \varphi\rangle$ stabilizes $Y_j$. Since $n=3$ is a proper divisor of $2f+1$, so $N<\langle \varphi\rangle$. It follows that $N$ is not the stabilizer of any $Y_j$. In particular, note that if (ii) occurs, then $\langle\varphi\rangle$ is the stabilizer of $Y_j$ with $j=(q^2+r+1)/5$.

We may therefore assume that exceptions (i) and (ii) does not occur.  Note that only one of $2f\pm n+1$ is divisible by $4$. So exactly one of $(q^2+r+1, q^2\pm 2^n)$ is greater than $1$ and in fact $\geqslant 5$ by Lemma \ref{2-5}. We may assume without loss that $(q^2+r+1, q^2-2^n)>1$. The same argument holds if we replace $q^2-2^n$ by $q^2+2^n$.  Let $j$ be the least positive integer such that $N$ is contained in the stabilizer of $Y_j$ in $\langle\varphi\rangle$. Then by Lemma \ref{3-2}, 
\[j=\frac{q^2+r+1}{(q^2+r+1, q^2-2^n)}.\]
Since
\[j\leqslant \frac{q^2+r+1}{5}<\frac{q^2+r}{4},\]
it follows that $Y_j\in\Irr(S)$. If the stabilizer, $T$, of $Y_j$ in $\langle\varphi\rangle$ properly contains $N$, then $T=\langle\varphi^t\rangle$ for  some divisor $t$ of $n$ with $1\leqslant t< n$. By Lemma \ref{3-2}, we have 
\[q^2+r+1\mid (q^2-2^t)j\qquad \text{or}\qquad q^2+r+1\mid (q^2+2^t)j.\]
It follows that one of $(q^2+r+1, q^2\pm 2^t)$ is greater than $1$. Again, we may assume that $(q^2+r+1, q^2+2^t)>1$. Then $(q^2+2^t)/(q^2+r+1, q^2-2^n)$ is an integer and so $(q^2+r+1, q^2-2^n)\mid q^2+2^t$. Thus  
\[j\geqslant \frac{q^2+r+1}{(q^2+r+1, q^2+2^t)}.\]
For equality, let $k=(q^2+r+1)/(q^2+r+1, q^2+2^t)$. Then $T$ and hence $N$ is contained in the stabilizer of $Y_k$ in $\langle\varphi\rangle$.  By the choice of  $j$,  we must have $j=k$, but this is impossible by Lemma \ref{2-6}. Hence $N=T$ and the proof is complete.
\end{proof}  
Finally we consider the characters $Z_k$, where $1\leqslant j\leqslant (q^2-r)/4$. Let $\varepsilon_2$ be a primitive $(q^2-r+1)$th root of unity. If $\xi_2$ is a generator of $A_2$, then
\begin{equation*}
Z_k({\xi_2}^l)=-(\varepsilon_2^{jl}+\varepsilon_2^{-jl}+\varepsilon_2^{jlq^2}+\varepsilon_2^{-jlq^2}).
\end{equation*}
Note that by \cite[Theorem ~13]{Suz}, the characters $Z_k$ are equal everywhere except on $A_2-\{1\}$.
Reasoning as in the proof of Lemmas \ref{3-2} and \ref{3-3}, we obtain the following result on the stabilizers of characters $Z_k$.
\begin{lem}\label{3-4}
Let $n$ be a positive divisor of $2f+1$. Set $N=\langle\varphi^n\rangle.$ Then $N$ is the stabilizer in $\langle\varphi\rangle$ of some $Z_k$ unless one of the following cases occur. 
\begin{itemize}
\item[{\rm (i)}]  $n=1$ and $f\equiv 0$ or $3 \,(\mod 4)$, 
\item[{\rm (ii)}] $n=3$ and $f\equiv 1$ or $2 \,(\mod 4)$.
\end{itemize}
In cases {\rm (i)} and {\rm (ii)}, $N$ does not stabilize any $Z_k$.
\end{lem}
\section{Main Results}
In this section, we complete the proof of Theorem A. It suffices by Lemma \ref{2-2} to find the degrees  of irreducible characters of $G$ with $S\leqslant G\leqslant \Aut(S)$ lying over characters $X_i$, $Y_j$ and $Z_k$ of $S$.
\begin{theorem}\label{4-1}
Let $S={^2B_2(q^2)}$, where $q^2=2^{2f+1}$, $f\geqslant 1$ and let $S\leqslant G\leqslant \Aut(S)$ with $|G : S|=d$. Then the degrees of the irreducible characters of $G$ lying over characters of $S$ of degree $q^4+1$ are precisely $(q^4+1)a$, where $a$ is a positive divisor of $d$, with the exception that $a\neq 1$ when $G=\Aut(S)$.
\end{theorem}

\begin{proof}
We have $G=S\langle\varphi^{(2f+1)/d}\rangle$. As it was mentioned in section 1, if $\theta\in\Irr(S)$, then $\cd(G\mid\theta)=\{|G : I_G(\theta)|\theta(1)\}.$ Thus, for $a\mid d$, there is a character of $G$ of degree $(q^4+1)a$ lying over a character $\theta$ of $S$ of degree $q^4+1$ if and only if 
\[I_G(\theta)=S\langle\varphi^n\rangle, \qquad\text{where $n=(\frac{2f+1}{d})a$}.\]
By Lemma \ref{3-1}, such a character $\theta$ of $S$ exists if and only if  $n\neq 1$. Since $d\mid 2f+1$, so $n=1$ if and only if $d=2f+1$ and $a=1$. Thus $a\neq 1$ whenever $G=\Aut(S)$.
\end{proof}
\begin{theorem}\label{4-2}
Let $S={^2B_2(q^2)}$, where $q^2=2^{2f+1}$, $f\geqslant 1$ and let $S\leqslant G\leqslant \Aut(S)$ with $|G : S|=d$. Then the degrees of the irreducible characters of $G$ lying over characters of $S$ of degree $(q^2-r+1)(q^2-1)$ are precisely $(q^2-r+1)(q^2-1)b$, where $b$ is a positive divisor of $d$, with the following exceptions:   
\begin{itemize}
\item[{\rm (i)}]  If $f\equiv 1$ or $2 \,\,(\mod 4)$, and $G=\Aut(S)$, then $b\neq 1$, 
\item[{\rm (ii)}] If $f\equiv 0$ or $3 \,\,(\mod 4)$, and $G=\Aut(S)$, then $b\neq 3$.
\end{itemize}
\end{theorem}

\begin{proof}
Reasoning as in the proof of Theorem \ref{4-1}, we note that for $b\mid d$, there is a character of $G$ of degree
$(q^2-r+1)(q^2-1)b$ lying over a character $\theta$ of $S$ of degree $(q^2-r+1)(q^2-1)$ if and only if 
\[I_G(\theta)=S\langle\varphi^n\rangle, \qquad\text{where $n=(\frac{2f+1}{d})b$}.\]
By Lemma \ref{3-3}, such a character $\theta$ of $S$ exists unless when 
\begin{itemize}
\item[(A1)]  $n=1$ and $f\equiv 1$ or $2 \,\,(\mod 4)$, or 
\item[(A2)] $n=3$ and $f\equiv 0$ or $3 \,\,(\mod 4)$.
\end{itemize}
In cases (A1) and (A2), $N$ does not stabilize any $Y_j$.

Suppose that (A1) holds. Since $d\mid 2f+1$, so $n=1$ if and only if $d=2f+1$ and $b=1$. Thus $b\neq 1$ whenever $f\equiv 1$ or $2\, (\mod 4)$ and $G=\Aut(S)$.

Now, suppose that (A2) holds. In this case, $n={(2f+1)b}/{d}=3$ implies that either $d=2f+1$ and $b=3$, or $d=(2f+1)/3$ and  $b=1$. If $d=2f+1$ and $b=3$, then $G=\Aut(S)$ and  exception (ii) follows. Also if $d=(2f+1)/3$ and $b=1$, then $G=S\langle\varphi^3\rangle$ and $I_G(\theta)=G$. By Lemma \ref{3-3}, there is a character $Y_j$ with $j=(q^2+r+1)/5$ invariant under $\varphi$, hence $I_G(Y_j)=G$. Thus $G$ has a character of degree $(q^2+r+1)(q^2-1)b$ with $b=1$. The proof is now complete.
\end{proof}
By a similar argument as above,  we obtain the following result on  characters of $S$  of degree $(q^2+r+1)(q^2-1)$.  
\begin{theorem}\label{4-3}
Let $S={^2B_2(q^2)}$, where $q^2=2^{2f+1}$, $f\geqslant 1$ and let $S\leqslant G\leqslant \Aut(S)$ with $|G : S|=d$. Then the degrees of the irreducible characters of $G$ lying over characters of $S$ of degree $(q^2+r+1)(q^2-1)$ are precisely $(q^2+r+1)(q^2-1)c$, where $c$ is a positive divisor of $d$, with the following exceptions: 
\begin{itemize}
\item[{\rm (i)}]  If $f\equiv 0$ or $3 \,(\mod 4)$, and $G=\Aut(S)$, then $c\neq 1$, 
\item[{\rm (ii)}] If $f\equiv 1$ or $2 \,(\mod 4)$, and $G=\Aut(S)$, then $c\neq 3$. 
\end{itemize}
\end{theorem}
Finally, note that  Theorem A follows immediately from Lemma \ref{2-2} and Theorems \ref{4-1}-\ref{4-3}.

\end{document}